\documentclass[12pt, leqno, twoside]{article}
\usepackage{amsmath, amsthm, amsfonts, amssymb, graphicx, color}

\numberwithin{equation}{section}
     \newtheorem{thm}{Theorem}[section]
     \newtheorem{cor}[thm]{Corollary}
     \newtheorem{prop}[thm]{Proposition}
     \newtheorem{lem}[thm]{Lemma}
\theoremstyle{definition}
     \newtheorem{defn}{Definition}[section]
     \newtheorem{exmp}{Example}[section]
\theoremstyle{remark}
     \newtheorem{rem}{Remark}[section]

\newcommand{\C}{\mathbb{C}}
\newcommand{\N}{\mathbb{N}}
\newcommand{\R}{\mathbb{R}}

\newcommand{\cY}{\mathcal{Y}}

\newcommand{\supp}{\operatorname{supp}}
\newcommand{\esssup}{\mathop{\mathrm{ess\,sup}}}

\newcommand{\Li}{L^{\infty}}

\newcommand{\LP}{L^{\Phi}}

\newcommand{\dlim}{\displaystyle\lim}

\newcommand{\dsup}{\displaystyle\sup}

\newcommand{\wL}{\mathrm{w}\hskip-0.6pt{L}}
\newcommand{\wLP}{\mathrm{w}\hskip-0.6pt{L}^{\Phi}}
\newcommand{\wLPs}{\mathrm{w}\hskip-0.6pt{L}^{\Psi}}

\newcommand{\iPy}{{\it{\Phi_Y}}}

\newcommand{\pwm}{\mathrm{PWM}}

\newcommand{\Lp}{L^p(\Omega)}
\newcommand{\Linfty}{L^{\infty}(\Omega)}

\newcommand{\msckw}{%
\footnotetext{\hspace{-0.35cm} 2010 {\it Mathematics Subject Classification}. 
46E30.
\endgraf{\it Key words and phrases.} 
Orlicz space, weak Orlicz space, pointwise multiplier. \par
Ryota Kawasumi, 
Department of Intelligent Systems Engineering, Ibaraki University, Hitachi, Ibaraki 316-8511, Japan,
rykawasumi@gmail.com  \par
Eiichi Nakai, Department of Mathematics, Ibaraki University, Mito, Ibaraki 310-8512, Japan,
eiichi.nakai.math@vc.ibaraki.ac.jp  
}
}

\title{Pointwise multipliers on weak Orlicz spaces \msckw}
\author{Ryota Kawasumi and Eiichi~Nakai}
\date{}

\begin{document}

\baselineskip=18pt

\maketitle

\begin{abstract}
We characterize the pointwise multipliers 
from a weak Orlicz space to another weak Orlicz space.
\end{abstract}

\section{Introduction}\label{sec:intro}

Let $\Omega=(\Omega,\mu)$ be a complete $\sigma$-finite measure space.
We denote by $L^0(\Omega)$ the set of
all measurable functions from $\Omega$ to $\R$ or $\C$.
Then $L^0(\Omega)$ is a linear space 
under the usual sum and scalar multiplication.
Let $E_1,E_2\subset L^0(\Omega)$ be subspaces.
We say that 
a function $g\in L^0(\Omega)$ is a pointwise multiplier from $E_1$ to $E_2$,
if the pointwise multiplication
$fg$ is in $E_2$ for any $f\in E_1$.
We denote by $\pwm(E_1,E_2)$ the set of all pointwise multipliers from $E_1$ to $E_2$.
We abbreviate $\pwm(E,E)$ to $\pwm(E)$.
For example,
\begin{equation*}
 \pwm(L^0(\Omega))=L^0(\Omega).
\end{equation*}
The pointwise multipliers are basic operators on function spaces
and thus the characterization of pointwise multipliers
is not only interesting itself but also sometimes very useful to other study.

For $p\in(0,\infty]$,
$\Lp$ denotes the usual Lebesgue space equipped with the norm
\begin{align*}
 \|f\|_{\Lp}
 &=\left(\int_{\Omega}|f(x)|^p\,d\mu(x)\right)^{1/p}, 
 \ \text{if} \ p\ne\infty,
\\
 \|f\|_{\Linfty}
 &=\esssup_{x\in\Omega}|f(x)|. 
\end{align*}
Then $\Lp$ is a complete quasi-normed space (quasi-Banach space).
If $p\in[1,\infty]$, then it is a Banach space.
It is well known as H\"older's inequality that
\begin{equation*}
 \|fg\|_{L^{p_2}(\Omega)}
 \le
 \|f\|_{L^{p_1}(\Omega)}
 \|g\|_{L^{p_3}(\Omega)},
\end{equation*}
for $1/p_2=1/p_1+1/p_3$ with $p_i\in(0,\infty]$, $i=1,2,3$.
This shows that
\begin{equation*}
 \pwm(L^{p_1}(\Omega),\,L^{p_2}(\Omega))\supset L^{p_3}(\Omega),
\end{equation*}
and 
\begin{equation*}
 \|g\|_{\pwm(L^{p_1}(\Omega),\,L^{p_2}(\Omega))}\le\|g\|_{L^{p_3}(\Omega)},
\end{equation*}
where $\|g\|_{\pwm(L^{p_1}(\Omega),\,L^{p_2}(\Omega))}$ 
is the operator norm of $g\in\pwm(L^{p_1}(\Omega),\,L^{p_2}(\Omega))$. 
Conversely, we can show the reverse inclusion
by using 
the uniform boundedness theorem or
the closed graph theorem.
That is,
\begin{equation}\label{pwm Lp123}
 \pwm(L^{p_1}(\Omega),L^{p_2}(\Omega))=L^{p_3}(\Omega)
 \quad\text{and}\quad
 \|g\|_{\pwm(L^{p_1}(\Omega),L^{p_2}(\Omega))}=\|g\|_{L^{p_3}(\Omega)}.
\end{equation}
If $p_1=p_2=p$, then
\begin{equation}\label{pwm Linfty}
 \pwm(\Lp)=\Linfty
 \quad\text{and}\quad
 \|g\|_{\pwm(L^{p}(\Omega))}=\|g\|_{\Linfty}.
\end{equation}
Proofs of \eqref{pwm Lp123} and \eqref{pwm Linfty} are
in Maligranda and Persson~\cite[Proposition~3 and Theorem~1]{Maligranda-Persson1989}.
See also \cite{Nakai2017LNA} for a survey on pointwise multipliers.
The characterization \eqref{pwm Lp123} was extended to several function spaces,
for example, Orlicz spaces, Lorentz spaces, Morrey spaces, etc, 
see \cite{CDS08, KLM2013,KLM2014,Lesnik-Tomaszewski2017,Maligranda-Nakai2010,Maligranda-Persson1989,
Nakai1996MOKU,Nakai1997MOKU,Nakai2000SCM,%
Nakai2016Nihonkai,Nakai2017springer} and the references in \cite{Nakai2017LNA}. 

In this paper we give the characterization of pointwise multipliers 
on weak Orlicz spaces.
To do this we first prove a generalized H\"older's inequality for the weak Orlicz spaces.
Next, to characterize the pointwise multipliers, we use the fact that 
all pointwise multipliers from a weak Orlicz space to another weak Orlicz space
are bounded operators.
This fact follows from 
Theorem~\ref{thm:g-Kantorovich} and Corollary~\ref{cor:g-Kantorovich} bellow.

We always assume that 
the function spaces $E\subset L^0(\Omega)$ have the following property,
see \cite[pages 94]{KantorovichAkilov1982}
in which this property is referred to as $\supp E=\Omega$:
\begin{multline}\label{supp Omega}
\text{If a measurable subset $\Omega_1\subset\Omega$ satisfies that}\\
\mu(\{x\in\Omega:f(x)\ne0\}\setminus\Omega_1)=0 \ 
\text{for every $f\in E$}, \\
\text{then} \ \mu(\Omega\setminus\Omega_1)=0.
\end{multline}
We say that a quasi-normed space $E\subset L^0(\Omega)$ has the lattice property
if the following holds:
\begin{equation}\label{lattice}
 f\in E, \ h\in L^0(\Omega), \ |h|\le|f| \text{ a.e.} 
 \ \ \Longrightarrow \ \ 
 h\in E, \ \|h\|_E\le\|f\|_E.
\end{equation}

Then we have the following theorem:

\begin{thm}[{\cite[Theorem~2.7]{Nakai2017LNA}}]\label{thm:g-Kantorovich}
Let a quasi-normed space $E\subset L^0(\Omega)$ have the lattice property \eqref{lattice}. 
For any sequence of functions $f_j\in E$, $j=1,2,\cdots$, 
if $f_j\to0$ in $E$, then
$f_j\to0$ in measure on every measurable set with finite measure.
\end{thm}

Using the closed graph theorem, we have the following corollary:

\begin{cor}[{\cite[Corollary~2.8]{Nakai2017LNA}}]\label{cor:g-Kantorovich}
If $E_1$ and $E_2$ are 
complete quasi-normed spaces with the lattice property \eqref{lattice}, 
then all $g\in\pwm(E_1,E_2)$ are bounded operators.
\end{cor}

Since the weak Orlicz spaces are 
complete quasi-normed spaces with the lattice property \eqref{lattice}, 
all pointwise multipliers from a weak Orlicz space to another weak Orlicz space
are bounded operators.

Orlicz spaces are introduced by \cite{Orlicz1932,Orlicz1936}.
For the theory of Orlicz spaces,
see \cite{Kita2009,Kokilashvili-Krbec1991,Krasnoselsky-Rutitsky1961,Maligranda1989,Rao-Ren1991}
for example.
See also \cite{Iaffei1996} for the weak Orlicz space.

The organization of this paper is as follows. 
We recall the definitions of the Young functions and the weak Orlicz spaces 
in Section~\ref{sec:Young}.
Then we state main results in Section~\ref{sec:result}. 
The proof method is the same as \cite{Maligranda-Nakai2010}.
However we need to investigate 
the properties of the quasi-norm on the weak Orlicz space.
We do this in Section~\ref{sec:prop}
to prove the main results in Section~\ref{sec:proof}.

\section{Young functions and weak Orlicz spaces}\label{sec:Young}

For an increasing function 
$\Phi:[0,\infty]\to[0,\infty]$,
let
\begin{equation*} 
 a(\Phi)=\sup\{t\ge0:\Phi(t)=0\}, \quad 
 b(\Phi)=\inf\{t\ge0:\Phi(t)=\infty\},
\end{equation*} 
with convention $\sup\emptyset=0$ and $\inf\emptyset=\infty$.
Then $0\le a(\Phi)\le b(\Phi)\le\infty$.

\begin{defn}[Young function]\label{defn:Young}
An increasing function $\Phi:[0,\infty]\to[0,\infty]$ is called a Young function 
(or sometimes also called an Orlicz function) 
if it satisfies the following properties;
\begin{align}\label{ab}
 &0\le a(\Phi)<\infty, \ 0<b(\Phi)\le\infty, \\
 &\lim_{t\to+0}\Phi(t)=\Phi(0)=0, \label{lim_0} \\
 &\text{$\Phi$ is convex $[0,b(\Phi))$}, \label{convex} \\
 &\text{if $b(\Phi)=\infty$, then } 
 \lim_{t\to\infty}\Phi(t)=\Phi(\infty)=\infty, \label{left cont infty} \\
 &\text{if $b(\Phi)<\infty$, then } 
 \lim_{t\to b(\Phi)-0}\Phi(t)=\Phi(b(\Phi)) \ (\le\infty). \label{left cont b}
\end{align}
\end{defn}

In what follows,
if an increasing and convex function $\Phi:[0,\infty)\to[0,\infty)$ satisfies
\eqref{lim_0} and $\dlim_{t\to\infty}\Phi(t)=\infty$,
then we always regard that $\Phi(\infty)=\infty$ and that $\Phi$ is a Young function.

We denote by $\iPy$ the set of all Young functions.
We also define three subsets $\cY^{(i)}$ ($i=1,2,3$) of $\iPy$ as 
\begin{align*}
 \cY^{(1)}
 &=
 \left\{\Phi\in\iPy:b(\Phi)=\infty\right\}, 
\\
 \cY^{(2)}
 &=
 \left\{\Phi\in\iPy:b(\Phi)<\infty,\ \Phi(b(\Phi))=\infty\right\}, 
\\
 \cY^{(3)}
 &=
 \left\{\Phi\in\iPy:b(\Phi)<\infty,\ \Phi(b(\Phi))<\infty\right\}.
\end{align*}

\begin{rem}\label{rem:Young}
We have the following properties of $\Phi\in\iPy$: 
\begin{enumerate}
\item[(Y1)]
If $\Phi\in\cY^{(1)}$, 
then $\Phi$ is absolutely continuous 
on any closed interval in $[0,\infty)$,
and $\Phi$ is bijective from $[a(\Phi),\infty)$ to $[0,\infty)$.
\item[(Y2)] 
If $\Phi\in\cY^{(2)}$, 
then $\Phi$ is absolutely continuous 
on any closed interval in $[0,b(\Phi))$,
and $\Phi$ is bijective from $[a(\Phi),b(\Phi))$ to $[0,\infty)$.
\item[(Y3)] 
If $\Phi\in\cY^{(3)}$, 
then $\Phi$ is absolutely continuous on $[0,b(\Phi)]$ 
and $\Phi$ is bijective from $[a(\Phi),b(\Phi)]$ to $[0,\Phi(b(\Phi))]$.
\item[(Y4)]
If $\Phi\in\cY^{(3)}$ and $0<\delta<1$, 
then there exists a Young function $\Psi\in\cY^{(2)}$ such that $b(\Phi)=b(\Psi)$ and
$$
 \Psi(\delta t) \le \Phi(t) \le \Psi(t) \quad\text{for all} \ t\in[0,\infty).
$$
To see this we only set $\Psi=\Phi+\Theta$, where
we choose $\Theta\in\cY^{(2)}$
such that $a(\Theta)=\delta\,b(\Phi)$ and $b(\Theta)=b(\Phi)$.
\end{enumerate}
\end{rem}

\begin{defn}\label{defn:LP}
For a Young function $\Phi$, let
\begin{align*}
  \LP(\Omega)
  &= \left\{ f\in L^0(\Omega):
     \int_{\Omega} \Phi(\epsilon |f(x)|)\,d\mu(x)<\infty
                 \;\text{for some}\; \epsilon>0 
    \right\}, \\
  \|f\|_{\LP(\Omega)} &=
  \inf\left\{ \lambda>0: 
    \int_{\Omega} \!\Phi\!\left(\frac{|f(x)|}{\lambda}\right) d\mu(x)
      \le 1
      \right\}, \\
  \wLP(\Omega)
  &= \left\{ f\in L^0(\Omega):
     \sup_{t\in(0,\infty)}\Phi(t)\,\mu(\epsilon f, t)<\infty
                 \;\text{for some}\; \epsilon>0 
    \right\}, \\
  \|f\|_{\wLP(\Omega)} &=
  \inf\left\{ \lambda>0: 
    \sup_{t\in(0,\infty)}\Phi(t)\,\mu\!\left(\frac{f}{\lambda}, t\right)
      \le 1
      \right\}, \\
  &\text{where} \quad \mu(f,t)=\mu\big(\{x\in\Omega:|f(x)|>t\}\big).
\end{align*}
\end{defn}
Then $\|\cdot\|_{\LP(\Omega)}$ is a norm and thereby $\LP(\Omega)$ is a Banach space,
and $\|\cdot\|_{\wLP(\Omega)}$ is a quasi-norm and thereby $\LP(\Omega)$ is a complete quasi-normed space
(quasi-Banach space).
For any Young function $\Phi$, 
\begin{equation*}
 \LP(\Omega)\subset\wLP(\Omega)
 \quad\text{with}\quad
 \|f\|_{\wLP(\Omega)}\le\|f\|_{\LP(\Omega)}.
\end{equation*}
Let
\begin{equation*}
 \Phi_{(\infty)}
 =
 \begin{cases}
  0,  & t\in[0,1], \\
  \infty, & t\in(1,\infty].
 \end{cases}
\end{equation*}
Then $\Phi_{(\infty)}$ is a Young function and
\begin{equation*}
 L^{\Phi_{(\infty)}}(\Omega)=\wL^{\Phi_{(\infty)}}(\Omega)=\Li(\Omega)
 \quad\text{with}\quad
 \|f\|_{L^{\Phi_{(\infty)}}(\Omega)}=\|f\|_{\wL^{\Phi_{(\infty)}}(\Omega)}=\|f\|_{\Li(\Omega)}.
\end{equation*}
If $\Phi$ be a Young function with $b(\Phi)<\infty$,
then $\Phi_{(\infty)}(t)\le\Phi(b(\Phi)t)$ for all $t\in[0,\infty]$.
Hence,
\begin{equation*}
 \wLP(\Omega)\subset\Li(\Omega)
 \quad\text{with}\quad
 \|f\|_{\Li(\Omega)}\le b(\Phi)\|f\|_{\wLP(\Omega)}.
\end{equation*}

We note that
\begin{equation}\label{weak norm equiv}
 \sup_{t\in(0,\infty)}t\,\mu(\Phi(|f|),t)
 =
 \sup_{t\in(0,\infty)}\Phi(t)\,\mu(f,t),
\end{equation}
and then
\begin{align*}
  \|f\|_{\wLP(\Omega)} &=
  \inf\left\{ \lambda>0: 
    \sup_{t\in(0,\infty)}\Phi(t)\,\mu\!\left(\frac{f}{\lambda}, t\right)
      \le 1
      \right\}, \\
  &=
  \inf\left\{ \lambda>0: 
    \sup_{t\in(0,\infty)}t\;\mu\!\left(\Phi\left(\frac{|f|}{\lambda}\right), t\right)
      \le 1
      \right\}.
\end{align*}
We give a proof of \eqref{weak norm equiv} for readers' convenience, 
see Proposition~\ref{prop:weak norm eq}.

Next we recall the generalized inverse of Young function $\Phi$
in the sense of O'Neil \cite[Definition~1.2]{ONeil1965}.

\begin{defn}\label{defn:ginverse}
For a Young function $\Phi$, let
\begin{equation}\label{inverse}
 \Phi^{-1}(u)
 = 
\begin{cases}
 \inf\{t\ge0: \Phi(t)>u\}, & u\in[0,\infty), \\
 \infty, & u=\infty.
\end{cases}
\end{equation}
\end{defn}
Then 
$\Phi^{-1}(u)$ is finite for all $u\in[0,\infty)$,
continuous on $(0,\infty)$ 
and right continuous at $u=0$.
If $\Phi$ is bijective from $[0,\infty]$ to itself, 
then $\Phi^{-1}$ is the usual inverse function of $\Phi$.

\begin{rem}\label{rem:g-inverse}
We have the following properties of $\Phi\in\iPy$
and its inverse:
\begin{enumerate}
\item[(P1)]
$\Phi(\Phi^{-1}(t)) \le t \le \Phi^{-1}(\Phi(t))$ for all $t\in[0,\infty]$
(Property 1.3 in \cite{ONeil1965}).
\item[(P2)]
$\Phi^{-1}(\Phi(t))=t$ if $\Phi(t)\in(0,\infty)$.
\item[(P3)]
If $\Phi\in\cY^{(1)}\cup\cY^{(2)}$, then
$\Phi(\Phi^{-1}(u))=u$ for all $u\in[0,\infty]$.
\end{enumerate}
\end{rem}

\begin{rem}\label{rem:inverse another}
Sometimes one defines
\begin{equation}\label{inverse another}
 \Phi^{-1}(u)
 = 
 \inf\{t\ge0: \Phi(t)>u\} \ (u\in[0,\infty))
 \quad\text{and}\quad
 \Phi^{-1}(\infty)=\lim_{u\to\infty}\Phi^{-1}(u).
\end{equation}
In this case 
$\Phi(\Phi^{-1}(u)) \le u$ for all $u\in[0,\infty)$ 
and $t \le \Phi^{-1}(\Phi(t))$ if $\Phi(t)\in[0,\infty)$.
\end{rem}

\section{Main results}\label{sec:result}

For Young functions $\Phi_1$ and $\Phi_2$,
we denote by $\|g\|_{\pwm(\wL^{\Phi_1}(\Omega),\wL^{\Phi_2}(\Omega))}$
the operator norm of $g\in\pwm(\wL^{\Phi_1}(\Omega),\wL^{\Phi_2}(\Omega))$.
The following result is a generalized H\"older's inequality for the weak Orlicz spaces.

\begin{thm}\label{thm:Holder}
Let $\Phi_i$, $i=1,2,3$, be Young functions.
If there exists a positive constant $C$ such that, for all $u\in(0,\infty)$,
\begin{equation}\label{assumption1}
 \Phi_1^{-1}(u)\Phi_3^{-1}(u)\le C\Phi_2^{-1}(u),
\end{equation}
then, for all $f\in\wL^{\Phi_1}(\Omega)$ and $g\in\wL^{\Phi_3}(\Omega)$,
\begin{equation*}
 \|fg\|_{\wL^{\Phi_2}(\Omega)}\le 4C\|f\|_{\wL^{\Phi_1}(\Omega)}\|g\|_{\wL^{\Phi_3}(\Omega)}.
\end{equation*}
Consequently,
\begin{equation*}
 \wL^{\Phi_3}(\Omega)\subset\pwm(\wL^{\Phi_1}(\Omega),\wL^{\Phi_2}(\Omega)),
\end{equation*}
and,
for all $g\in L^{\Phi_3}(\Omega)$,
\begin{equation*}
 \|g\|_{\pwm(\wL^{\Phi_1}(\Omega),\wL^{\Phi_2}(\Omega))}\le 4C\|g\|_{\wL^{\Phi_3}(\Omega)}.
\end{equation*}
\end{thm}

For the Orlicz spaces, it is known by O'Neil~\cite{ONeil1965} that, 
if \eqref{assumption1} holds, 
then 
\begin{equation*}
 \|fg\|_{L^{\Phi_2}(\Omega)}\le 2C\|f\|_{L^{\Phi_1}(\Omega)}\|g\|_{L^{\Phi_3}(\Omega)}.
\end{equation*}

Next, we state our main result.

\begin{thm}\label{thm:converse}
Let $\Phi_i$, $i=1,2,3$, be Young functions.
If there exists a positive constant $C$ such that, for all $u\in(0,\infty)$,
\begin{equation}\label{assumption2}
 \Phi_2^{-1}(u)\le C\Phi_1^{-1}(u)\Phi_3^{-1}(u),
\end{equation}
then
\begin{equation*}
 \pwm(\wL^{\Phi_1}(\Omega),\wL^{\Phi_2}(\Omega))\subset \wL^{\Phi_3}(\Omega),
\end{equation*}
and, 
for all $g\in\pwm(\wL^{\Phi_1}(\Omega),\wL^{\Phi_2}(\Omega))$,
\begin{equation*}
 \|g\|_{\wL^{\Phi_3}(\Omega)}
 \le
 C\|g\|_{\pwm(\wL^{\Phi_1}(\Omega),\wL^{\Phi_2}(\Omega))}.
\end{equation*}
\end{thm}

In \cite{Maligranda-Nakai2010}
it was shown that, 
if \eqref{assumption2} holds, 
then 
\begin{equation*}
 \pwm(L^{\Phi_1}(\Omega),L^{\Phi_2}(\Omega))\subset L^{\Phi_3}(\Omega),
\end{equation*}
and, 
for all $g\in\pwm(L^{\Phi_1}(\Omega),L^{\Phi_2}(\Omega))$,
\begin{equation*}
 \|g\|_{L^{\Phi_3}(\Omega)}
 \le
 C\|g\|_{\pwm(L^{\Phi_1}(\Omega),L^{\Phi_2}(\Omega))}.
\end{equation*}

\begin{cor}\label{cor:pwm wO}
Let $\Phi_i$, $i=1,2,3$, be Young functions.
If there exists a positive constant $C_i$, $i=1,2$, such that, for all $u\in(0,\infty)$,
\begin{equation*}
 C_1^{-1}\Phi_2^{-1}(u)
 \le
 \Phi_1^{-1}(u)\Phi_3^{-1}(u)
 \le
 C_2\Phi_2^{-1}(u),
\end{equation*}
then
\begin{equation*}
 \pwm(\wL^{\Phi_1}(\Omega),\wL^{\Phi_2}(\Omega))=\wL^{\Phi_3}(\Omega),
\end{equation*}
and
\begin{equation*}
 C_1^{-1}\|g\|_{\wL^{\Phi_3}(\Omega)}
 \le
 \|g\|_{\pwm(\wL^{\Phi_1}(\Omega),\wL^{\Phi_2}(\Omega))}
 \le
 4C_2\|g\|_{\wL^{\Phi_3}(\Omega)}.
\end{equation*}
\end{cor}

In the following, for functions $P,Q:[0,\infty)\to[0,\infty)$,
$P(t)\sim Q(t)$ means that
there exists a positive constant $C$ such that
$C^{-1}P(t)\le Q(t)\le CP(t)$ for all $t\in[0,\infty)$.

\begin{exmp}\label{exmp:w LplogLq}
Let $p_i,q_i\in[1,\infty)$, $i=1,2,3$, and
\begin{equation*}
 \Phi_i(t)=t^{p_i}\max(1,\log t)^{q_i}, \quad i=1,2,3.
\end{equation*}
Then
\begin{equation*}
 \Phi_i^{-1}(t)
 \sim
 t^{1/p_i}\max(1,\log t)^{-q_i/p_i}, \quad i=1,2,3.
\end{equation*}
Hence, if $1/p_1+1/p_3=1/p_2$ and $q_1/p_1+q_3/p_3=q_2/p_2$,
then 
\begin{equation*}
 \pwm(\wL^{\Phi_1}(\Omega),\wL^{\Phi_2}(\Omega))=\wL^{\Phi_3}(\Omega),
\end{equation*}
and $\|g\|_{\pwm(\wL^{\Phi_1}(\Omega),\wL^{\Phi_2}(\Omega))}$ 
is comparable to $\|g\|_{\wL^{\Phi_3}(\Omega)}$.
\end{exmp}

\begin{exmp}\label{exmp:w expLp}
Let $p_i,q_i\in[1,\infty)$, $i=1,2,3$, and
\begin{equation*}
 \Phi_i(t)=\exp(t^{p_i})-1, \quad i=1,2,3.
\end{equation*}
Then
\begin{equation*}
 \Phi_i^{-1}(t)
 \sim
 \begin{cases}
  t^{1/p_i}, & 0\le t<2, \\
  (\log t)^{1/p_i}, & 2\le t<\infty,
 \end{cases}
 \quad i=1,2,3.
\end{equation*}
Hence, if $1/p_1+1/p_3=1/p_2$,
then 
\begin{equation*}
 \pwm(\wL^{\Phi_1}(\Omega),\wL^{\Phi_2}(\Omega))=\wL^{\Phi_3}(\Omega),
\end{equation*}
and $\|g\|_{\pwm(\wL^{\Phi_1}(\Omega),\wL^{\Phi_2}(\Omega))}$ 
is comparable to $\|g\|_{\wL^{\Phi_3}(\Omega)}$.
\end{exmp}

\section{Properties of the quasi-norm}\label{sec:prop}

In this section we investigate the properties of the quasi-norm $\|\cdot\|_{\wLP(\Omega)}$
to prove the main results.

For two Young functions $\Phi$ and $\Psi$,
if there exist positive constants $C_1$ and $C_2$ such that
\begin{equation*}
 \Phi(C_1t)\le\Psi(t)\le\Phi(C_2t) \quad\text{for all $t\in[0,\infty]$},
\end{equation*}
then 
$\wLP(\Omega)=\wLPs(\Omega)$ and
\begin{equation*}
 C_1\|f\|_{\wLP(\Omega)}\le\|f\|_{\wLPs(\Omega)}\le C_2\|f\|_{\wLP(\Omega)}.
\end{equation*}

By the measure theory we have the following property:
\begin{equation}\label{meas converge}
 f_j\ge0 \ \text{and} \ f_j\nearrow f \ \text{a.e.}
 \quad\Longrightarrow\quad 
 \lim_{j}\mu(f_j,t)=\mu(f,t) \ \text{for each $t\in[0,\infty)$}.
\end{equation}
From this property 
and the left continuity of the Young function $\Phi$
we have the following property:
\begin{equation}\label{le1}
   \sup_{t\in(0,\infty)} t\,\mu\!\left( \Phi\left(\frac{|f(x)|}{\|f\|_{\wLP(\Omega)}}\right),t\right)\le1.
\end{equation}
We also have the Fatou property: 
\begin{multline}\label{Fatou}
 f_j\in \wLP(\Omega) \ (j=1,2\cdots), \ f_j\ge0, \ f_j\nearrow f \ \text{a.e. and} 
  \ \sup_j\|f_j\|_{\wLP(\Omega)}<\infty,
 \\
 \Longrightarrow\quad
 f\in \wLP(\Omega) \ \text{and} \ \|f\|_{\wLP(\Omega)}\le\sup_j\|f_j\|_{\wLP(\Omega)}.
\end{multline}

\begin{prop} \label{prop:simple=1}
If $\Phi\in\cY^{(1)}\cup\cY^{(2)}$ 
and $g$ is a finitely simple function and $g\not=0$, 
then $g\in\wLP(\Omega)$ and
\begin{equation*}
 \sup_{t\in(0,\infty)} t\,\mu\!\left( \Phi\left(\frac{|g(\cdot)|}{\|g\|_{\wLP(\Omega)}}\right),t\right)=1.
\end{equation*}
\end{prop}

\begin{proof}
Let 
\begin{equation*}
 I_{\Phi}(g)
 =
 \sup_{t\in(0,\infty)} t\,\mu(\Phi(|g(\cdot)|),t).
\end{equation*}

{\bf Case 1.}
$\Phi\in\cY^{(1)}$:
In this case $\Phi$ is strictly increasing and bijective from $(a(\Phi),\infty)$ to $(0,\infty)$.
Let $g$ be a finitely simple function. 
We may assume that $g\geq0$, i.e.,
\begin{equation*}
 g = \sum_{k=1}^N c_k \chi_{A_k},
 \quad
 0=c_0 < c_1 < c_2 < ... < c_N < \infty,\ 0 < \mu(A_k) < \infty,
\end{equation*}
where $A_k$ are pairwise disjoint. 
Then every $\Phi(c_k/\lambda)$ is continuous and nonincreasing with respect to $\lambda > 0$. 
Moreover, $\Phi(c_k/\lambda)$ is strictly decreasing on $(0, c_k/a(\Phi))$ 
(for $a(\Phi) = 0$ we understand $c_k/a(\Phi) = \infty$). 
Observing
\begin{multline*}
 \mu\,\left(\left\{x\in\Omega:\Phi\left(\frac{g(x)}{\lambda}\right)>t\right\}\right)
 =
 \mu\,\left(\left\{x\in\Omega:\sum_{k=1}^N \Phi\left(\frac{c_k}{\lambda}\right)\chi_{A_k}>t\right\}\right) \\
 =
 \sum_{k=j}^{N}\mu(A_k),
 \quad\text{if}\ \
 \Phi\left(\frac{c_{j-1}}{\lambda}\right)\le t<\Phi\left(\frac{c_j}{\lambda}\right),
 \quad j=1,2,\dots,N,
\end{multline*}
we have
\begin{equation*}
 I_{\Phi}\left(\frac{g}{\lambda}\right)
 =
 \sup_{t\in(0,\infty)} t\,\mu\!\left( \Phi\left(\frac{g}{\lambda}\right),t\right)
 =
 \max_{1\le j\le N} \Phi\left(\frac{c_j}{\lambda}\right)  \sum_{k=j}^{N}\mu(A_k).
\end{equation*}
Therefore, $I_{\Phi}(g/\lambda)$
is continuous and strictly decreasing on $(0, c_N/a(\Phi))$. 
Since $\lim_{\lambda \rightarrow 0} I_{\Phi}(g/\lambda) = \infty$ 
and $\lim_{\lambda \rightarrow c_N/a(\Phi)} I_{\Phi}(g/\lambda) = 0$, 
we obtain that $ I_{\Phi}(g/\cdot)$ is bijective from $(0, c_N/a(\Phi))$ to $(0, \infty)$.
That is, there exists a unique $\lambda\in(0, c_N/a(\Phi))$ such that $I_{\Phi}(g/\lambda) = 1$.

{\bf Case 2.}
$\Phi\in\cY^{(2)}$:
In this case $\Phi$ is strictly increasing and bijective from $(a(\Phi),b(\Phi))$ to $(0,\infty)$.
Let $g$ be a simple function as in Case~1.
Then, in the same way as in Case~1, we obtain that
$I_{\Phi}(g/\cdot)$ is bijective from $(c_N/b(\Phi), c_N/a(\Phi))$ to $(0, \infty)$.
That is, there exists a unique $\lambda\in(c_N/b(\Phi), c_N/a(\Phi))$ such that $I_{\Phi}(g/\lambda) = 1$.
\end{proof}

In the rest of this section 
we show the following proposition.

\begin{prop}\label{prop:weak norm eq}
For any Young function $\Phi$,
\begin{equation}\label{weak norm eq}
 \sup_{t\in(0,\infty)} \Phi(t)\, \mu(f,t)
 =
 \sup_{u\in(0,\infty)} u\, \mu(f,\Phi^{-1}(u))
 =
 \sup_{u\in(0,\infty)} u\, \mu(\Phi(|f(\cdot)|),u).
\end{equation}
\end{prop}

\begin{rem}\label{rem:weak norm eq}
If $t=u=0$, 
then
\begin{equation*}
 \Phi(t)\, \mu(f,t)
 =
 u\, \mu(f,\Phi^{-1}(u))
 =
 u\, \mu(\Phi(|f(\cdot)|),u)
 =0,
\end{equation*}
since $\Phi(0)=0$.
If $t=u=\infty$, then 
\begin{equation*}
 \{x:|f(x)|>t\}
 =
 \{x:|f(x)|>\Phi^{-1}(u)\}
 =
 \{x:\Phi(|f(\cdot)|)>u\}
 =\emptyset,
\end{equation*}
since $\Phi^{-1}(\infty)=\infty$,
that is,
\begin{equation*}
 \Phi(t)\, \mu(f,t)
 =
 u\, \mu(f,\Phi^{-1}(u))
 =
 u\, \mu(\Phi(|f(\cdot)|),u)
 =0.
\end{equation*}
\end{rem}

\begin{lem}\label{lem:weak norm eq}
Let $\Phi$ be a Young function with $a(\Phi)<b(\Phi)$.
If $u\in(0,\Phi(b(\Phi)))$,
then
\begin{equation*}
 \{x:|f(x)|>\Phi^{-1}(u)\}
 =
 \{x:\Phi(|f(x)|)>u\}.
\end{equation*}
\end{lem}

\begin{proof}
Let $a=a(\Phi)$ and $b=b(\Phi)$.
Then 
$\Phi$ is bijective from $(a,b)$ to $(0,\Phi(b))$
in any case of 
$b<\infty$ or $b=\infty$;
$\Phi(b)<\infty$ or $\Phi(b)=\infty$.
Let $t=\Phi^{-1}(u)$. Then 
\begin{equation*}
 t\in(a,b) \Leftrightarrow u\in(0,\Phi(b)).
\end{equation*}
If $|f(x)|\in(a,b)$,
then 
\begin{equation*}
 |f(x)|>t\Leftrightarrow\Phi(|f(x)|)>\Phi(t).
\end{equation*}
That is,
\begin{equation*}
 |f(x)|>\Phi^{-1}(u)\Leftrightarrow\Phi(|f(x)|)>u.
\end{equation*}
If $|f(x)|\le a$,
then 
\begin{equation*}
 |f(x)|\le a<t=\Phi^{-1}(u)
 \quad\text{and}\quad
 \Phi(|f(x)|)=0<u.
\end{equation*}
If $|f(x)|\ge b$,
then 
\begin{equation*}
 |f(x)|\ge b>t=\Phi^{-1}(u)
 \quad\text{and}\quad
 \Phi(|f(x)|)\ge\Phi(b)>u.
\end{equation*}
Therefore, we have the conclusion.
\end{proof}


\begin{proof}[Proof of Proposition~\ref{prop:weak norm eq}]
Let $a=a(\Phi)$ and $b=b(\Phi)$.

\noindent
{\bf Case 1:} 
Let $\Phi\in\cY^{(1)}\cup\cY^{(2)}$. 
Then $\Phi$ is bijective from $(a,b)$ to $(0,\infty)$,
and then
\begin{align*}
 \sup_{t\in(0,b)} \Phi(t)\, \mu(f,t)
 &=
 \sup_{t\in(a,b)} \Phi(t)\, \mu(f,t) \\
 &=
 \sup_{u\in(0,\infty)} u\, \mu(f,\Phi^{-1}(u)) \\
 &=
 \sup_{u\in(0,\infty)} u\, \mu(\Phi(|f(\cdot)|),u),
\end{align*}
where we used Lemma~\ref{lem:weak norm eq} for the last equality.
If $b=\infty$, then the above equalities show \eqref{weak norm eq}.
If $b<\infty$, then
\begin{equation*}
 \sup_{t\in[b,\infty)} \Phi(t)\, \mu(f,t)=0 \quad\text{or}\quad \infty.
\end{equation*}
If $\dsup_{t\in[b,\infty)} \Phi(t)\, \mu(f,t)=0$, then 
\begin{equation*}
 \sup_{t\in(0,\infty)} \Phi(t)\, \mu(f,t)=\sup_{t\in(0,b)} \Phi(t)\, \mu(f,t).
\end{equation*}
If $\dsup_{t\in[b,\infty)} \Phi(t)\, \mu(f,t)=\infty$, then $\mu(f,b)>0$
and $\dlim_{t\to b-0}\Phi(t)\,\mu(f,t)=\infty$.
Hence
\begin{equation*}
 \sup_{t\in(0,\infty)} \Phi(t)\, \mu(f,t)
 =
 \sup_{t\in(0,b)} \Phi(t)\, \mu(f,t)
 =\infty.
\end{equation*}
Therefore, we have \eqref{weak norm eq}.

\noindent
{\bf Case 2:} 
Let $\Phi\in\cY^{(3)}$ and $a<b$.
Then $\Phi$ is bijective from $(a,b)$ to $(0,\Phi(b))$, 
and then
\begin{align*}
 \sup_{t\in(0,b)} \Phi(t)\, \mu(f,t)
 &=
 \sup_{t\in(a,b)} \Phi(t)\, \mu(f,t) \\
 &=
 \sup_{u\in(0,\Phi(b))} u\, \mu(f,\Phi^{-1}(u)) \\
 &=
 \sup_{u\in(0,\Phi(b))} u\, \mu(\Phi(|f(\cdot)|),u),
\end{align*}
where we used Lemma~\ref{lem:weak norm eq} for the last equality.
If $\mu(f,b)=0$, then 
\begin{equation*}
 \mu(f,\Phi^{-1}(u))=\mu(\Phi(|f(\cdot)|),u)=0 \quad\text{for} \ u\in[\Phi(b),\infty),
\end{equation*}
since $\Phi^{-1}(u)=b$ and
\begin{equation*}
 \{x:\Phi(|f(x)|)>u\}\subset\{x:\Phi(|f(x)|)>\Phi(b)\}\subset\{x:|f(x)|>b\}.
\end{equation*}
Hence,
\begin{equation*}
 \sup_{t\in[b,\infty)} \Phi(t)\, \mu(f,t)
 =
 \sup_{u\in[\Phi(b),\infty)} u\, \mu(f,\Phi^{-1}(u))
 =
 \sup_{u\in[\Phi(b),\infty)} u\, \mu(\Phi(|f(\cdot)|),u)
 =0.
\end{equation*}
If $\mu(f,b)>0$,
then $\mu(f,b+1/j)>0$ for some $j\in\N$ by the measure theory.
Hence,
\begin{equation*}
 \sup_{t\in(0,\infty)} \Phi(t)\, \mu(f,t)\ge
 \Phi(b+1/j)\mu(f,b+1/j)=\infty.
\end{equation*}
On the other hand,
$\mu(f,b)>0$ implies that,
for all $u\in(\Phi(b),\infty)$,
\begin{equation*}
 \mu(f,\Phi^{-1}(u))=\mu(f,b)>0,
 \quad
 \mu(\Phi(|f(\cdot)|),u)\ge\mu(\{x:\Phi(|f(x)|)=\infty\})>0
\end{equation*}
Hence,
\begin{equation*}
 \sup_{u\in(0,\infty)} u\, \mu(f,\Phi^{-1}(u))= \sup_{u\in(0,\infty)} u\, \mu(\Phi(|f(\cdot)|),u)=\infty.
\end{equation*}
Therefore, we have \eqref{weak norm eq}.

\noindent
{\bf Case 3:} 
Let $\Phi\in\cY^{(3)}$ and $a=b$.
Then $\Phi(t)=0$ for $t\in(0,b]$ and
$\Phi^{-1}(u)=b$ for $u\in(0,\infty)$.
If $\mu(f,b)=0$, then 
$|f(x)|\le b$ and $\Phi(|f(x)|)=0$ a.e.\,$x$.
Hence.
\begin{equation*}
 \sup_{t\in(0,\infty)} \Phi(t)\, \mu(f,t)
 =
 \sup_{u\in(0,\infty)} u\, \mu(f,\Phi^{-1}(u))
 =
 \sup_{u\in(0,\infty)} u\, \mu(\Phi(|f(\cdot)|),u)
 =0.
\end{equation*}
If $\mu(f,b)>0$, then, by the same way as Case~2,
we have
\begin{equation*}
 \sup_{t\in(0,\infty)} \Phi(t)\, \mu(f,t)
 =
 \sup_{u\in(0,\infty)} u\, \mu(f,\Phi^{-1}(u))
 =
 \sup_{u\in(0,\infty)} u\, \mu(\Phi(|f(\cdot)|),u)
 =\infty.
\end{equation*}
Therefore, we have \eqref{weak norm eq}.
\end{proof}

\section{Proofs}\label{sec:proof}

\begin{proof}[Proof of Theorem~\ref{thm:Holder}]\label{proop:thm:Holder}
Let $f\in\wL^{\Phi_1}(\Omega)$ and $g\in\wL^{\Phi_3}(\Omega)$.
We may assume that $f,g\ge0$ and $\|f\|_{\wL^{\Phi_1}(\Omega)}=\|g\|_{\wL^{\Phi_3}(\Omega)}=1$.
Let
\begin{equation*}
 h(x)=\max\big(\Phi_1(f(x)),\Phi_3(g(x))\big).
\end{equation*}
Then, by the assumption \eqref{assumption1} and (P1),
\begin{align*}
 f(x)g(x)
 &\le 
 \Phi_1^{-1}(\Phi_1(f(x)))\, \Phi_3^{-1}(\Phi_3(g(x))) \\
 &\le 
 \Phi_1^{-1}(h(x))\, \Phi_3^{-1}(h(x))
 \le 
 C\Phi_2^{-1}(h(x)).
\end{align*}
Hence, by (P1),
\begin{equation*}
 \Phi_2\left(\frac{f(x)g(x)}C\right)
 \le \Phi_2(\Phi_2^{-1}(h(x)))
 \le h(x)
 \le \Phi_1(f(x))+\Phi_3(g(x)).
\end{equation*}
Then
\begin{align*}
 &\sup_{t\in(0,\infty)} t\,\mu\!\left(\Phi_2\left(\frac{f(x)g(x)}{4C}\right),t\right) \\
 &\le
 \sup_{t\in(0,\infty)} t\,\mu\!\left(\frac14\Phi_2\left(\frac{f(x)g(x)}{C}\right),t\right) \\
 &=
 \frac12\sup_{t\in(0,\infty)} t\,\mu\!\left(\Phi_2\left(\frac{f(x)g(x)}{C}\right),2t\right) \\
 &\le
 \frac12\sup_{t\in(0,\infty)} t\,\mu\big(\Phi_1(f(x))+\Phi_3(g(x)),2t\big) \\
 &\le
 \frac12\sup_{t\in(0,\infty)} t\,\bigg(\mu\big(\Phi_1(f(x)),t\big)+\mu\big(\Phi_3(g(x)),t\big)\bigg) \\
 &\le
 \frac12(1+1)=1,
\end{align*}
where we used \eqref{le1} for the last inequality.
Therefore,
$\|fg\|_{\wL^{\Phi_2}(\Omega)}\le4C$
and the proof is complete.
\end{proof}

\begin{proof}[Proof of Theorem~\ref{thm:converse}]\label{proop:thm:converse}
{\bf Case 1.}
$\Phi_2$ and $\Phi_3$ are in $\cY^{(1)}\cup\cY^{(2)}$:
Let 
\begin{equation*}
 g\in\pwm(\wL^{\Phi_1}(\Omega),\wL^{\Phi_2}(\Omega)).
\end{equation*}
Assume first that $g$ is a finitely simple function.
Then $g \in\wL^{\Phi_3}(\Omega)$ and
\begin{equation*}
 G(x) := \Phi_3 \left (\frac{|g(x)|}{\|g\|_{\wL^{\Phi_3}(\Omega)}} \right) < \infty 
 \quad\text{a.e. in}\ \Omega.
\end{equation*}
Put 
\begin{equation*}
 h(x) = 
 \begin{cases}
  \Phi_1^{-1}(G(x)) & \text{if}\ 0<G(x)<\infty, \\ 
  0                 & \text{if}\ G(x)=0.
 \end{cases}
\end{equation*}
From the property (P1) it follows that $\Phi_1(h(x)) \leq G(x)$ a.e. in $\Omega$ and
\begin{equation*}
 \sup_{t\in(0,\infty)}t\,\mu\Big({\Phi_1}(h),t\Big) 
 \le
 \sup_{t\in(0,\infty)}t\,\mu\!\left({\Phi_3}\left(\frac{|g(\cdot)|}{\|g\|_{\wL^{\Phi_3}(\Omega)}}\right),t\right)
 \le 1,
\end{equation*}
which gives $\|h\|_{\wL^{\Phi_1}(\Omega)} \le 1$.
Next we show that
\begin{equation}\label{Phi1>Phi3}
 \Phi_2\left( C h(x) \frac{|g(x)|}{\|g\|_{\wL^{\Phi_3}(\Omega)}} \right) 
 \ge
  G(x)=\Phi_3\left( \frac{|g(x)|}{\|g\|_{\wL^{\Phi_3}(\Omega)}} \right).
\end{equation}
If $G(x)>0$, then by the property (P2) and the assumption \eqref{assumption2},
\begin{align*}
 h(x)\, \frac{|g(x)|}{\|g\|_{\wL^{\Phi_3}(\Omega)}} 
 &=
 \Phi_1^{-1}(G(x))\,\Phi_3^{-1}\left( \Phi_3\left(\frac{|g(x)|}{\|g\|_{\wL^{\Phi_3}(\Omega)}}\right) \right) \\ 
 &= 
 \Phi_1^{-1}(G(x))\,\Phi_3^{-1}(G(x)) 
 \ge
 \frac{1}{C}\Phi_2^{-1}(G(x))
\end{align*}
and hence, by (P3),
\begin{equation*}
 \Phi_2\left( C h(x) \frac{|g(x)|}{\|g\|_{\wL^{\Phi_3}(\Omega)}} \right) 
 \ge
 \Phi_2\left(\Phi_2^{-1}(G(x)) \right) 
 =
 G(x).
\end{equation*}
If $G(x) = 0$, then $h(x) = 0$ and 
$\Phi_2 \left( C h(x) \frac{|g(x)|}{\|g\|_{\wL^{\Phi_3}(\Omega)}} \right) = 0$.
Thus, we have \eqref{Phi1>Phi3}.
By Proposition~\ref{prop:simple=1} we have
\begin{equation*}
 \sup_{t\in(0,\infty)}t\,
  \mu\!\left({\Phi_2}\left( C h(\cdot) \frac{|g(\cdot)|}{\|g\|_{\wL^{\Phi_3}(\Omega)}}\right),t\right)
 \ge
 \sup_{t\in(0,\infty)}t\,
    \mu\!\left({\Phi_3}\left( \frac{|g(\cdot)|}{\|g\|_{\wL^{\Phi_3}(\Omega)}}\right),t\right)
 =
 1
\end{equation*}
and so $\| h g\|_{\wL^{\Phi_2}(\Omega)} \ge \frac1C{\| g \|_{\wL^{\Phi_3}(\Omega)}}$, 
that is,  
\begin{equation*}
 \|g\|_{\pwm(\wL^{\Phi_1}(\Omega),\wL^{\Phi_2}(\Omega))} 
 \ge
 \frac{1}{C} \,\|g\|_{\wL^{\Phi_3}(\Omega)},
\end{equation*}
where we use the fact that the pointwise multiplier $g$ is a bounded operator.

In the general case, $g$ can be approximated by a sequence of finitely simple functions $\{g_j\}$ such that
$0\le g_j\nearrow|g|$ a.e. in $\Omega$,
since $\mu$ is a $\sigma$-finite measure. Then
\begin{equation*}
 \|g\|_{\pwm(\wL^{\Phi_1}(\Omega),\wL^{\Phi_2}(\Omega))}
 \ge
 \|g_j\|_{\pwm(\wL^{\Phi_1}(\Omega),\wL^{\Phi_2}(\Omega))}
 \ge
 \frac{1}{C} \,\|g_j\|_{\wL^{\Phi_3}(\Omega)}
\end{equation*}
by our first part of the proof. 
Using the Fatou property \eqref{Fatou} of the quasi-norm $\|\cdot\|_{\wL^{\Phi_3}(\Omega)}$,
we obtain
\begin{equation*}
 \|g\|_{\pwm(\wL^{\Phi_1}(\Omega),\wL^{\Phi_2}(\Omega))}
 \ge
 \frac{1}{C} \,\|g\|_{\wL^{\Phi_3}(\Omega)}.
\end{equation*}

{\bf Case 2.} 
$\Phi_2\in\cY^{(3)}$ or $\Phi_3\in\cY^{(3)}$:
We consider only the case that both $\Phi_2$ and $\Phi_3$ are in $\cY^{(3)}$, 
since other cases are similar.
In this case, by (Y4), for all $0<\delta<1$, 
there exist $\Psi_2\in\cY^{(2)}$ and $\Psi_3\in\cY^{(2)}$ such that
\begin{equation*}\label{Phi-delta}
\Psi_2(\delta u) \le \Phi_2(u) \le \Psi_2(u), 
\quad 
\Psi_3(\delta u) \le \Phi_3(u) \le \Psi_3(u) 
\quad\text{for all $u$}.
\end{equation*}
It follows that
\begin{gather*}
 \delta\Phi_2^{-1}(u)\le\Psi_2^{-1}(u)\le\Phi_2^{-1}(u),
 \quad 
 \delta\Phi_3^{-1}(u)\le\Psi_3^{-1}(u)\le\Phi_3^{-1}(u), \\
 \delta\|g\|_{\wL^{\Psi_2}(\Omega)}\le\|g\|_{\wL^{\Phi_2}(\Omega)}\le\|g\|_{\wL^{\Psi_2}(\Omega)},
 \quad 
 \delta\|g\|_{\wL^{\Psi_3}(\Omega)}\le\|g\|_{\wL^{\Phi_3}(\Omega)}\le\|x\|_{\wL^{\Psi_3}(\Omega)}, 
\end{gather*}
and
\begin{equation*}
 \delta\|g\|_{\pwm(\wL^{\Phi_1}(\Omega),\wL^{\Psi_2}(\Omega))}
 \le
 \|g\|_{\pwm(\wL^{\Phi_1}(\Omega),\wL^{\Phi_2}(\Omega))}
 \le
 \|g\|_{\pwm(\wL^{\Phi_1}(\Omega),\wL^{\Psi_2}(\Omega))}.
\end{equation*}
Using the inequality
\begin{equation*}
 \Psi_2^{-1}(u)\le \frac{C}{\delta} \Phi_1^{-1}(u)\Psi_3^{-1}(u), 
\end{equation*}
which follows by \eqref{assumption2} and the definitions of $\Psi_2$ and $\Psi_3$, we have
\begin{equation*}
 \|g\|_{\pwm(\wL^{\Phi_1}(\Omega),\wL^{\Psi_2}(\Omega))}
 \ge
 \frac{\delta}{C} \,\|g\|_{\wL^{\Psi_3}(\Omega)}
\end{equation*}
by Case~1. 
Then
\begin{equation*}
 \|g\|_{\pwm(\wL^{\Phi_1}(\Omega),\wL^{\Phi_2}(\Omega))}
 \ge
 \frac{\delta^2}{C} \,\| g\|_{\wL^{\Phi_3}(\Omega)}
\end{equation*}
holds for all $0<\delta<1$.
Therefore,
\begin{equation*}
 \|g\|_{\pwm(\wL^{\Phi_1}(\Omega),\wL^{\Phi_2}(\Omega))}
 \ge
 \frac{1}{C} \,\|g\|_{\wL^{\Phi_3}(\Omega)},
\end{equation*}
and the proof is finished. 
\end{proof}

\section*{Acknowledgement}
The authors would like to thank Professors~Hiro-o Kita and Takashi Miyamoto for their useful comments.
The second author was supported by Grant-in-Aid for Scientific Research (B), 
No.~15H03621, Japan Society for the Promotion of Science.




\end{document}